\documentclass[11pt,letterpaper]{amsart}
\usepackage{amssymb}
\usepackage{color}
\usepackage[linkcolor=blue,citecolor=blue, colorlinks=true,backref=true,bookmarks=true]{hyperref}

\theoremstyle{plain}
\newtheorem{theorem}{Theorem}[section]

\newtheorem{lemma}[theorem]{Lemma}
\newtheorem{proposition}[theorem]{Proposition}
\newtheorem{definition}[theorem]{Definition}

\theoremstyle{definition}
\newtheorem{remark}[theorem]{Remark}

\def\R{\mathbb{R}}
\def\B{\mathbb{B}}
\def\S{\mathbb{S}}

\def\Z{\mathbb{Z}}

\numberwithin{equation}{section}

\title{Stable free boundary minimal hypersurfaces in a wedge domain}

\begin{document}

\author{Zetian Yan}
\address{Department of Mathematics \\ UC Santa Barbara \\ Santa Barbara \\ CA 93106 \\ USA}
\email{ztyan@ucsb.edu}
\begin{abstract}
 We prove that a stable $C^{1,1}$-to-edge properly embedded free boundary minimal hypersurface $\Sigma^3$ of a $4$-dimensional wedge domain $\Omega^4_{\theta}$ with angle $\theta\in (0,\pi]$ is flat. 
\end{abstract}
\maketitle

\section{Introduction}
{\textbf{Classical Bernstein problem.}} For an entire minimal graph, i.e., the graph $\Gamma(u)$ of a function $u:\R^n\to \R$ satisfying the minimal hypersurface equation, does it have to be necessarily a hyperplane in $\R^{n+1}$?

In 1914, Sergei Bernstein solved this problem in the case $n=2$. In 1962, Fleming \cite{Flembing62} provided an alternative proof. In other low dimensional cases, the classical Bernstein problem was solved by De Giorgi \cite{De65}(n=3), Almgren \cite{Almgren66}(n=4) and Simons \cite{Simons68} ($5\leqslant n\leqslant 7$), respectively. Note that a minimal graph $\Gamma(u)$ in $\R^{n+1}$ satisfy the volume growth condition: for the ball $B_r(0)$ in $\R^{n+1}$, we have
\begin{equation}\label{growth}
    \mathrm{Vol}(B_r(0)\cap \Gamma(u))\leqslant \frac{\mathrm{Vol}(\S^{n})}{2}r^{n},
\end{equation}
where $\S^{n}$ is the unit sphere in $\R^{n+1}$. Moreover, two-dimensional minimal graphs satisfy the area-minimizing property automatically. Combining these with the log-cutoff trick yields the flatness; see \cite[Chapter 1]{CM11} for more details.

A natural generalization of the classical Bernstein problem is the following:

{\textbf{Stable Bernstein problem.}} Is every complete orientable immersed stable minimal hypersurface a hyperplane? 

When $n=2$, the stable Bernstein problem was confirmed by do Carmo and Peng \cite{doP79}, Fischer-Colbrie and Schoen \cite{FDS80} and Pogorelov \cite{Pogorelov81}, respectively. As for $n\leqslant 5$, it is also true with some additional assumptions; see \cite{SSY75,doP82,CSZ97,Chen01,NS07} and references therein. In particular, under the assumption on the volume growth of geodesic balls, Schoen, Simon and Yau \cite{SSY75} proved the flatness by estimating the $L^p$ bounds on the second fundamental form.

In high dimensional cases, both classical and stable Bernstein problems are not true. For $n\geqslant8$, Bombieri, De Giorgi and Giusti \cite{BDG69} showed that there are minimal entire graphs that are not hyperplanes. Moreover, in $\R^8$, there are non-flat area-minimizing complete orientable hypersurfaces constructed in \cite{HS85}. 

The stable Bernstein problem in the remaining cases without additional hypothesis were still open. In 2021, Chodosh and Li \cite{CL21} gave the positive answer in $\R^4$ by estimating the quantity
\begin{equation*}
    F(t)=\int_{\Sigma_t}|\nabla u|^2
\end{equation*}
and relating it to
\begin{equation*}
    A(t)=\int_{\Sigma_t}|A_M|^2,
\end{equation*}
where $u$ is the minimal positive Green's function, $\Sigma_t$ is the $t$-level set of $u$ and $A_M$ is the second fundamental form of $M$. Later, they gave an alternative proof \cite{CL23} by using the $\mu$-bubble to control volume growth and combining it with the argument in \cite{SSY75}. Recently, Chodosh, Li, Minter and Stryker \cite{CLMS24} generalized the $\mu$-bubble argument and solved the stable Bernstein problem in $\R^5$. In $\R^4$, we also note that 
Catino, Mastrolia, Roncoroni \cite{CMR24} provided a different proof by the conformal transformation.

It is interesting to consider Bernstein-type problems in less regular domains. In this paper, we consider the free boundary stable Bernstein problem in a wedge domain $\Omega^m_{\theta}$ with angle $\theta\in (0,\pi]$, where the wedge domain $\Omega^m_{\theta}$ is in the form of
\begin{equation*}
    \Omega^m=\Omega^2_{\theta}\times \R^{m-2}=\mathrm{Clos}\left(\left\{x\in \R^m: x_1>0, x_2\in (\tan (-\frac{\theta}{2})x_1, \tan (\frac{\theta}{2})x_1)\right\}\right).
\end{equation*} 
In the 1990s, Hildebrandt and Sauvigny \cite{HS97a,HS97b,HS99a,HS99b} were among the first to investigate free boundary minimal surfaces in a wedge domain $\Omega^2_{\theta}$. From their work, we note that the wedge angle $\theta$ affects the behavior of minimal surfaces crucially. In particular, when $\theta=\pi$, by the reflection principle \cite{GLZ20}, the free boundary stable Bernstein problem in $\R^{n+1}_+$ is equivalent to that in $\R^{n+1}$. While, for $\theta\in (0,\pi)$, the global reflection principle is no longer available, and the regularity of free boundary minimal hypersurfaces in $\Omega^m_{\theta}$ is subtle.

For $3\leqslant n+1\leqslant 6$, Mazurowski and Wang \cite{MW23} proved a Bernstein-type theorem under the assumption on the volume growth (\ref{growth}). We use the $\mu$-bubble technique to obtain the volume estimate and confirm this problem without additional assumptions in dimension $4$.
\begin{theorem}\label{main}
    Suppose that $\left(\Sigma^3, \left\{\partial_3 \Sigma\right\}\right)\subset \left(\Omega^4, \left\{\partial_4 \Omega\right\}\right)$ is a stable $C^{1,1}$-to-edge properly embedded free boundary minimal hypersurface. There exists an explicit constant $C$ such that
    \begin{equation}
        |\Sigma\cap B_{\R^4}(0,\rho)|_g\leqslant C \rho^3.
    \end{equation}
    In particular, $\Sigma=\Omega\cap P$ where $P\in \R^4$ is a hyperplane.
\end{theorem}

We adapt the construction of weighted free boundary $\mu$-bubbles in \cite{CL23}. First of all, in the same spirit in \cite{CL23}, we carry out the conformal deformation technique used by Gulliver--Lawson \cite{GL86}. By the local reflection argument, we find that free boundaries are totally geodesic with respect to the blow-up metric. Besides, in our setting, the weighted function $u$ should satisfy the Neumann condition on the boundary and have $C^{1,\alpha}$ regularity near the edge. 

Note that in \cite[Lemma 21]{CL23}, the properly embeddedness of free boundary $\mu$-bubbles comes from the maximum principle because both constrained boundaries and free boundary $\mu$-bubbles are minimizers of the same functional. However, due to the existence of the edge, the properly embeddedness here is a delicate issue. Suppose that $\Xi$ is a stable free boundary $\mu$-bubble in $\Sigma$. Touching phenomenon consists of two cases \cite{MW23}:
\begin{itemize}
    \item the interior $\mathring{\Xi}$ touches the face $\partial^F \Sigma$;
    \item the face $\partial^F \Xi$ touches the edge $\partial^E \Sigma$.
\end{itemize}
By the Neumann condition on $u$ and the totally geodesic property of free boundaries, we find that the first case cannot happen. As for the second case, if happened, we can construct a vector field along which the first variation of $\Xi$ is negative, which violates the stability of $\Xi$.

Finally, combining the log-cutoff trick and the fact that free boundaries are totally geodesic, boundary integrals in the second variation of a $\mu$-bubble vanish and we obtain the volume estimate following \cite{CL23}.

The remainder of the paper is organized as follows. In Section 2, we review several basic concepts which is borrowed from \cite{MW23}. In Section 3, by the local reflection argument, we prove the one-end theorem following \cite{CSZ97}. In Section 4 and 5, we construct a free boundary $\mu$-bubble and obtain the volume control.

{\bf{Acknowledgements.}} I would like to thank Dr. Tongrui Wang, Dr. Liangjun Weng, Nan Wu and Dr. Xingyu Zhu for their encouragement and useful discussion. Thank Dr. Yaoting Gui for reviewing the draft.
\section{Preliminary}
\subsection{Locally wedge-shaped hypersurfaces.}
To begin with, let us fix some notations in the Euclidean space. Let $H^m_+$ and $H^m_{-}$ be two closed half spaces in $\R^m$, where $m\in \{2,3, \cdots\}$. $\Omega^m:=H^m_+ \cap H^m_-$ is called an $m$-dimensional {\textit{wedge domain}} if it has non-empty interior. By rotating, we can always write $\Omega^m$ in the form of
\begin{equation*}
    \Omega^m=\Omega^2_{\theta}\times \R^{m-2}=\mathrm{Clos}\left(\left\{x\in \R^m: x_1>0, x_2\in (\tan (-\frac{\theta}{2})x_1, \tan (\frac{\theta}{2})x_1)\right\}\right),
\end{equation*}
where $\theta\in (0,\pi]$ is called the {\textit{wedge angle}} of $\Omega^m$. We usually use the notation $\Omega^m_{\theta}$ to emphasize its angle. 

\begin{definition}[Stratification]
    For a non-trivial $m$-dimensional wedge domain $\Omega^m_{\theta}$; i.e. $\theta\in (0,\pi)$, define the stratification $\partial_{m-2}\Omega\subset \partial_{m-1}\Omega \subset \partial_{m}\Omega$ of $\Omega^m_{\theta}$ by
    \begin{equation*}
        \partial_{m}\Omega:=\Omega, \quad \partial_{m-1}\Omega:=\partial \Omega,\quad \partial_{m-2}\Omega:=
\{0\}\times \R^{m-2}.
    \end{equation*}
Then we define
\begin{itemize}
    \item $\mathring{\Omega}:=\partial_{m}\Omega\backslash \partial_{m-1}\Omega$ to be the interior of $\Omega$;
    \item $\partial^{F}\Omega:=\partial_{m-1}\Omega\backslash \partial_{m-2}\Omega$ to be the face of $\Omega$;
    \item $\partial^{E}\Omega:=\partial_{m-2}\Omega$ to be the edge of $\Omega$.
\end{itemize}
\end{definition}
Next, we introduce the notations for hypersurfaces that are locally modeled by wedge domains.
\begin{definition}[Locally wedge-shaped hypersurfaces]
We say that $M^n\subset \Omega^{n+1}_{\theta}$ is an embedded locally wedge-shaped hypersurface of $\Omega^{n+1}_{\theta}$, if for any $p\in M^n$, there exists $R=R(p)>0$ and a diffeomorphism $\psi=\psi_p: \B^{n+1}_R(0)\to \B^{n+1}_R(p)$ so that 
\begin{itemize}
    \item $\psi(0)=p$ and the tangent map $(D\psi)_0\in O(n+1)$ is an orthogonal transformation;
    \item $\psi\left(\left(\Omega(p)\times \{0\}\right)\cap \B^{n+1}_R(0)\right)=M\cap\B^{n+1}_R(p) $, where
    \begin{equation*}
        \Omega(p)=\left\{
        \begin{array}{cc}
           \R^n,  & p\in \mathring{M}^n, \\
            \Omega^n_{\theta}(p), & p\in \partial M^n,
        \end{array}
        \right.
    \end{equation*}
    for some $n$-dimensional wedge domain $\Omega^n_{\theta}(p)$ with $\theta(p)\in (0,\pi]$.
\end{itemize}
We call $(\psi_p, \B^{n+1}_R(p), \Omega(p))$ a local model of $M^n$ around $p$, and call $\theta$ the wedge angle of $M^n$ at $p\in \partial M$. Additionally, given $l\in \Z_+$ and $\alpha\in (0,1]$, if for any $p\in \partial M$ with $\theta(p)\in (0,\pi)$, $\psi_p$ is globally a $C^{l,\alpha}$-diffeomorphism, and is a $C^{\infty}$-diffeomorphism in $(\R^n\times\{0\})\backslash \partial^E\Omega(p)$, then we say $M^n$ is a $C^{l,\alpha}$-to-edge locally wedge-shaped $n$-hypersurface; see \cite[Section 2]{MW23} for more details.
\end{definition}
Similarly, an embedded locally wedge-shaped hypersurface can be stratified using its local models.
\begin{definition}
    Let $M^n\subset \Omega^{n+1}$ is an embedded locally wedge-shaped hypersurface of $\Omega^{n+1}$. We define the stratification $\partial_{m-2}M\subset \partial_{m-1}M\subset \partial_{m}M$ of $M$ by
    \begin{equation*}
        \partial_{n}M:=M, \quad \partial_{n-1}M:=\partial M,\quad \partial_{n-2}M:=\bigcup_{p\in \partial M}\psi_p\left(\left(\partial_{n-2}\Omega^n_{\theta}(p)\times \{0\}\right)\cap \B_R^{n+1}(0)\right),
    \end{equation*}
where $(\psi_p, \B^{n+1}_R(p), \Omega(p))$ is a local model of $M^n$ around $p$. Moreover, we define
\begin{itemize}
    \item $\mathring{M}:=\partial_{n}M\backslash \partial_{n-1}M$ to be the interior of $M$;
    \item $\partial^{F}M:=\partial_{n-1}M\backslash \partial_{n-2}M$ to be the face of $M$;
    \item $\partial^{E}M:=\partial_{n-2}M$ to be the edge of $M$.
\end{itemize}
\end{definition}
Using the stratification of locally wedge-shaped hypersurfaces, we now introduce the definition of almost properly embedding.
\begin{definition}[Almost properly embedding]
   Let $M^n\subset \Omega^{n+1}$ is an embedded locally wedge-shaped hypersurface of $\Omega^{n+1}$. We say that $M^n$ is almost properly embedded in $\Omega^{n+1}$, denoted by $(M^n, \{\partial_n M^n\})\subset (\Omega^{n+1}, \{\partial_{n+1}\Omega^{n+1}\})$, if
   \begin{equation*}
       \partial_i M^n\subset \partial_{i+1}\Omega^{n+1}, \quad i\in \{n-2, n-1,n\}.
   \end{equation*}
   In particular, if $\partial^F M^n=\partial^F \Omega\cap M $ and $\partial^E M^n=\partial^E \Omega\cap M$, then we say that $(M^n, \{\partial_n M^n\})\subset (\Omega^{n+1}, \{\partial_{n+1}\Omega^{n+1}\})$ is properly embedded.
\end{definition}
In the sequel, we will only consider properly embedded locally wedge-shaped hypersurfaces.

\subsection{Variations for locally wedge-shaped hypersurfaces.}
We now consider the variational problem for free boundary minimal hypersurfaces (FBMHs) properly embedded in $\Omega^{n+1}_{\theta}$. Let $\Sigma^n\subset \Omega^{n+1}$ is an embedded locally wedge-shaped hypersurface of $\Omega^{n+1}$. By Stokes' theorem, the first variation formula of the area functional is given by
\begin{equation*}
    \delta \mathcal{A}_{\Sigma}(X)=-\int_{\Sigma}\langle H_{\Sigma}, X\rangle \mathrm{dvol}_{\Sigma} +\int_{\partial^F \Sigma}\langle \eta, X\rangle \mathrm{dvol}_{\partial ^F\Sigma},
\end{equation*}
where $\eta$ denotes the unit outward co-normal along $\partial^F \Sigma$ and $X$ is an admissible vector field defined in \cite[Definition 3.3]{MW23}. Moreover, we have the following characterization.
\begin{proposition}[\cite{MW23}]
    A properly embedded hypersurface $(\Sigma^n, \{\partial_n \Sigma^n\})\subset (\Omega^{n+1}, \{\partial_{n+1}\Omega^{n+1}\})$ is a FBMH if and only if the following two conditions hold:
    \begin{itemize}
        \item $H_{\Sigma}\equiv 0$ in $\mathring{\Sigma}$,
        \item $\eta\perp \partial^F \Omega$ on $\partial^F \Sigma$.
    \end{itemize}
\end{proposition}
Next, we investigate the stability of locally wedge-shaped hypersurfaces. Suppose that $\Sigma$ is two-sided so that it admits a continuous choice of normal vector $\nu$.
\begin{proposition}[\cite{MW23}]
    Assume that $\Sigma$ is a FBMH. Suppose that the admissible vector field $X$ restricts to vector field $f\nu$ on $\Sigma$. Then the second variation of the area functional is given by
    \begin{equation}\label{second}
        \delta^2 \mathcal{A}_{\Sigma}(X)=\int_{\Sigma} \left(|\nabla f|^2-|A|^2f^2\right)\mathrm{dvol}_{\Sigma}
    \end{equation}
    where $A$ denotes the second fundamental form of $\Sigma$ and $\eta$ is the unit outward co-normal along $\partial \Sigma$.
\end{proposition}
\begin{remark}
    The boundary integral in the original form vanishes due to the properly embedding and a reflection technique used in \cite[Theorem 5.3]{MW23}.
\end{remark}

Similar to \cite[Lemma 1.36]{CM11}, the stability of $\Sigma$:
\begin{equation}
        \delta^2 \mathcal{A}_{\Sigma}(X)=\int_{\Sigma} \left(|\nabla f|^2-|A|^2f^2\right)\mathrm{dvol}_{\Sigma}\geqslant 0
    \end{equation}
is equivalent to the existence of a positive function $w$ on $\Sigma$ satisfying
\begin{equation}\label{stability}
\begin{split}
\left(\Delta_{\Sigma}+|A|^2\right)(w)&=0, ~~\mbox{in}~~ \Sigma, \\
\frac{\partial w}{\partial {\eta}}&=0, ~~\mbox{on}~~ \partial \Sigma.
    \end{split}
\end{equation}
By the well-known construction, the universal cover of $\Sigma$ is also a FBMH because the canonical projection map is locally isometric. Moreover, if $\Sigma$ is a stable FBMH, the universal cover $\overline{\Sigma}$ is also stable because the lifting of $w$ satisfies (\ref{stability}) on $\overline{\Sigma}$. Without loss of generality, in the sequel, we assume that $\Sigma$ is simply connected.  

\section{One-endness} 
Suppose that $\Omega^{n+1}=H^{n+1}_+\cap H^{n+1}_{-}$ is a $n+1$-dimensional wedge domain in $\R^{n+1}$ with wedge angle $\theta$, and $\left(\Sigma^n, \left\{\partial_n \Sigma\right\}\right)\subset \left(\Omega^{n+1}, \left\{\partial_{n+1} \Omega\right\}\right)$ is a stable $C^{1,1}$-to-edge properly embedded FBMH. We define
\begin{equation*}
    \partial^{F}_{+}\Sigma:=\partial^{F}\Sigma\cap \partial H^{n+1}_+, \quad \partial^{F}_{-}\Sigma:=\partial^{F}\Sigma\cap \partial H^{n+1}_-.
\end{equation*}
for all $p\in \partial^F \Sigma$. We can assume that $p\in \partial^F_+ \Sigma$. Let $\tilde{\Sigma}$ be the hypersurface consisting of $\Sigma$ together with its reflection across $\partial H^{n+1}_{+}$. Then $\tilde{\Sigma}$ is a smooth minimal hypersurface with no boundary in a neighborhood of $p$ in $\R^{n+1}$.

We claim that $\partial^{F}_{+}\Sigma$ and $\partial^{F}_{-}\Sigma$ are totally geodesic in $(\Sigma, g)$ with induced metric $g$. We extend $e_i\in T_p \partial^{F}_{+}\Sigma$, $i=\{1,\cdots, n\}$, to be an orthonormal frame in a neighborhood of $p$ in $\tilde{\Sigma}$. Let $\gamma:(-1,1)\to \tilde{\Sigma}$ be a curve that
\begin{equation*}
    \gamma(0)=p,\quad \gamma^{'}(0)=\eta,\quad  \gamma(t)=R(\gamma(-t)),
\end{equation*}
where $R$ denotes reflection across $\partial H^{n+1}_{+}$. Note that the function
\begin{equation*}
    f(t):=\langle \nabla^g_{e_i}e_j, \gamma^{'}\rangle (\gamma(t)), \quad i,j=\{1,\cdots, n\},
\end{equation*}
is an odd function, which implies that for $i,j=\{1,\cdots, n\}$
\begin{equation*}
    f(0)=\langle \nabla^g_{e_i}e_j, \eta\rangle=0.
\end{equation*}
Therefore, we conclude that $\partial^{F}_{+}\Sigma$ and $\partial^{F}_{-}\Sigma$ are totally geodesic. 
 
By analyzing harmonic functions on $\Sigma^n$, we will show that $\Sigma^n$ has only one end. In the sequel, without loss of generality, we may assume that $0\in \partial^{E}\Sigma^n$.

\begin{lemma}\label{volume}
    Every end of $\Sigma^n$ has infinite volume. 
\end{lemma}
\begin{proof}
    We let $r(x)=\mathrm{dist}_{\R^{n+1}}(0,x)$ be the distance function of $\R^{n+1}$, and $\bar{r}(x)=\mathrm{dist}_{\Sigma}(0,x)$ be the distance function on $\Sigma$ with respect to the induced metric. Choosing $s_0$ large enough so that 
\begin{equation*}
    \Sigma \backslash \B_{\R^{n+1}}(s_0)= \bigcup_{j=1}^k E_j, \quad k \geqslant 1
\end{equation*}
is the disjoint union of noncompact connected components. By Sard's theorem, we can assume that $\B_{\R^{n+1}}(s_0)$ intersects with each $E_j$ transversally. Without loss of generality, we take $E=E_1$ as an example. There are three mutually exclusive possibilities: $E$ admits the boundary which 
\begin{enumerate}
    \item doesn't intersect with $\partial H^{n+1}_+\cup \partial H^{n+1}_{-}$;
    \item intersects with $\partial H^{n+1}_+$ or $\partial H^{n+1}_{-}$ merely;
    \item intersects with both $\partial H^{n+1}_+$ and $\partial H^{n+1}_{-}$.
\end{enumerate}
In the second case, without loss of generality, we can assume that $E$ intersects with $\partial H^{n+1}_+$. Let $\tilde{E}$ be the hypersurface consisting of $E$ together with its reflection across $\partial H^{n+1}_{+}$. For convenience, in the first and third case, let $\tilde{E}$ be $E$.

Let $\B_{\Sigma}(s)$ be the geodesic ball contained in $\tilde{E}$, of radius $s$ centered at the origin.
By Lemma 1 in \cite{CSZ97}, we know that $s^{-n}\mathrm{Vol}(B_{\Sigma}(s))$ is non-decreasing. Therefore
\begin{equation*}
    \frac{\mathrm{Vol}(B_{\Sigma}(s))}{s^n}\geqslant \lim_{s\to 0} \frac{\mathrm{Vol}(B_{\Sigma}(s))}{s^n}=\omega(n),
\end{equation*}
where $\omega(n)$ is the volume of unit ball in $\R^n$. If $E$ has finite volume, it implies that $\tilde{E}$ has finite volume as well. In the first and second cases, we can choose $R$ big enough such that
\begin{equation*}
    \mathrm{Vol}(\tilde{E})\geqslant \mathrm{Vol}(B_{p}(s))\geqslant \omega(n)R^n> \mathrm{Vol}(\tilde{E}),
\end{equation*}
a contradiction.

We follow the argument in \cite[Lemma 1]{CSZ97} to deal with the third case.  As shown in \cite{CSZ97}, we have
\begin{equation}\label{bound}
    \frac{\partial r}{\partial \bar{r}}\leqslant 1.
\end{equation}
By a direct computation and using the fact that $\Sigma$ is minimal, we can show that in the interior of $\Sigma$,
\begin{equation}\label{distance}
    \Delta_{\Sigma} r^2(x)=2n.
\end{equation}
Note that the radial function $r^2(x)$ can be extended to be an even function on $\tilde{\Sigma}$ around any point $p$ on $\partial^F \Sigma$. Similarly, we can conclude that
\begin{equation}\label{vanish}
    \frac{\partial r^2}{\partial \eta_{\pm}}(p)=0 ~~\mbox{on}~~ \partial H^{n+1}_{\pm},
\end{equation}
where $\eta_{\pm}$ are the unit outward co-normal along $\partial H^{n+1}_{\pm}$.

Due to (\ref{vanish}), integrating (\ref{distance}) over $\B_{\Sigma}(s)\cap E$ yields
\begin{equation}\label{int}
    2n\mathrm{vol}(\B_{\Sigma}(s)\cap E)=\int_{\partial \B_{\Sigma}(s)\cap \mathring{E}} \frac{\partial r^2}{\partial \eta_1} \mathrm{dvol}+\int_{\partial E\cap \overline{\B}_{\R^{n+1}}(s_0)} \frac{\partial r^2}{\partial \eta_2} \mathrm{dvol},
\end{equation}
where $\eta_1$ and $\eta_2$ are the unit outward co-normal vector fields along $\partial \B_{\Sigma}(s)\cap \mathring{E}$ and $\partial E\cap \overline{\B}_{\R^{n+1}}(s_0)$, respectively.

By our assumption, $\B_{\R^{n+1}}(s_0)$ intersects with $E$ transversally. It implies that $\eta_2$ points inward the ball $\B_{\R^{n+1}}(s_0)$ along $\partial E\cap \overline{\B}_{\R^{n+1}}(s_0)$. Therefore,
\begin{equation}\label{negative}
    \int_{\partial E\cap \overline{\B}_{\R^{n+1}}(s_0)} \frac{\partial r^2}{\partial \eta_2} \mathrm{dvol}\leqslant 0.
\end{equation}

Similarly, we have
\begin{equation*}
    2n\mathrm{vol}(\B_{\Sigma}(s)\cap E)\leqslant 2s \mathrm{vol}(\partial \B_{\Sigma}(s)\cap \mathring{E}).
\end{equation*}
Note that 
\begin{equation*}
    \mathrm{vol}(\partial \B_{\Sigma}(s)\cap \mathring{E})=\frac{\partial}{\partial t}\Big|_{t=s}\mathrm{vol}(\B_{\Sigma}(t)\cap E).
\end{equation*}
In the same spirit as \cite[Lemma 1]{CSZ97}, we conclude that the quantity $s^{-n}\mathrm{vol}( \B_{\Sigma}(s)\cap \mathring{E})$ is nondecreasing. Therefore,
\begin{equation*}
    \frac{\mathrm{vol}( \B_{\Sigma}(s)\cap \mathring{E})}{s^n}\geqslant \frac{\mathrm{vol}( \B_{\Sigma}(s_0)\cap \mathring{E})}{s_0^n}>0,
\end{equation*}
for fixed value $s_0$. If $E$ has finite volume, similar to the above, we obtain a contradiction.
\end{proof}

\begin{lemma}\label{harmonic}
Suppose that $\Sigma^n$ has at lease two ends, then here exists a non-constant bounded harmonic function with finite energy on $\Sigma^n$.
\end{lemma}
\begin{proof}
As constructed in Lemma 2 \cite{CSZ97}, there is an exhausion $\left\{D_i\right\}$ of $\Sigma$ by compact submanifolds with boundary. For $i\geqslant i_0$ and $i_0$ sufficiently large, let
\begin{equation*}
    \Sigma\backslash D_i=\bigcup_{j=1}^{n}E^{(i)}_j,  \quad n \geqslant 1
\end{equation*}
be the disjoint union of connected components.
By Lemma \ref{volume} and the assumption, $\Sigma$ has at least two components with infinite volume. Let $E^{(i_0)}_1$ and $E^{(i_0)}_2$ be two such components. On each compact domain $D_i$, we consider the following Dirichlet--Neumann problem
\begin{equation}\label{DNP}
\left\{
\begin{array}{ll}
\Delta_{\Sigma} u &=0, ~~\mbox{in}~~ \mathring{D_i} \\
u &=1, ~~\mbox{on}~~\partial E^{(i)}_1\\
u &=0, ~~\mbox{on}~~\partial E^{(i)}_j, j\neq 1\\
\frac{\partial u}{\partial \eta}&=0, ~~\mbox{on}~~\partial D_i\cap \partial^F\Sigma,
\end{array}
\right.
\end{equation}
where $\eta$ is the unit conormal to $\partial^F\Sigma$. Let $u_i$ be the unique solution of (\ref{DNP}). By the maximum principle and Hopf lemma, we have $0\leqslant u_i\leqslant 1$ on $D_i$. Moreover, it's easy to see that for $i>j$
\begin{equation*}
    \int_{D_i} |\nabla u_i|^2\mathrm{dvol}_{\Sigma}\leqslant \int_{D_j} |\nabla u_j|^2\mathrm{dvol}_{\Sigma}.
\end{equation*}
For simplicity, in the following arguments, let $C>0$ be a varying uniform constant. Hence, there is a constant $C$ such that 
\begin{equation*}
    \int_{D_i} |\nabla u_i|^2\mathrm{dvol}_{\Sigma}<C.
\end{equation*}
Therefore by passing to a subsequence, still denoted by $u_i$, we can find a harmonic function $u$ on $\Sigma$ satisfying
\begin{equation*}
    \lim_{i\to \infty}u_i=u, \quad \int_{\Sigma} |\nabla u|^2\mathrm{dvol}_{\Sigma}<C, \quad 0\leqslant u\leqslant 1.
\end{equation*}
In the following we prove that the limiting harmonic function $u$ is not a constant function. We will prove this by contradiction.

Setting $\phi=u_i(1-u_i)$. Note that from the construction of $u_i$, $\phi$ vanishes on $\partial E^{(i)}_1\cup \partial E^{(i)}_2$ and has finite energy as well. Besides, $\phi$ satisfies the Neumann condition on $\partial D_i\cap \partial^F\Sigma$. We define the Sobolev trace quotient $\mathcal{Q}_{\partial}(\Sigma)$ by
\begin{equation*}
    \mathcal{Q}_{\partial}(\Sigma)=\inf \left\{\mathcal{Q}^{g}_{\partial}(\varphi):\varphi\in C^1(\Sigma), \varphi\neq 0, \frac{\partial \varphi}{\partial \eta_{\pm}}=0\right\},
    \end{equation*}
where 
\begin{equation*}
    \mathcal{Q}^{g}_{\partial}(\varphi):=\frac{\int_{\Sigma}\left(|\nabla \varphi|^2_g+\frac{n-2}{4(n-1)}R_g \varphi^2\right)\mathrm{dvol}_{\Sigma}+\int_{\partial \Sigma}\left(\varphi\frac{\partial \varphi}{\partial \eta}+\frac{n-2}{2}H_{\partial \Sigma}\varphi^2\right)\mathrm{dvol}_{\partial \Sigma}}{\left(\int_{\partial \Sigma}|\varphi|^{\frac{2(n-1)}{n-2}}\mathrm{dvol}_{\partial \Sigma}\right)^{\frac{n-2}{n-1}}}.
\end{equation*}
The boundary integral in the numerator of $\mathcal{Q}^{g}_{\partial}(\varphi)$ vanishes because $\varphi$ satisfies the Neumann condition and $\partial^F \Sigma$ is totally geodesic. Moreover, by the stability of $\Sigma^n$ and the Gauss equation, we have
\begin{equation*}
    \int_{\Sigma}\left(|\nabla \varphi|^2_g+\frac{n-2}{4(n-1)}R_g \varphi^2\right)\mathrm{dvol}_{\Sigma}>\int_{\Sigma}\left(|\nabla \varphi|^2_g-|A|^2 \varphi^2\right)\mathrm{dvol}_{\Sigma}\geqslant0,
\end{equation*}
which implies that $ \mathcal{Q}_{\partial}(\Sigma)>0$. 

Therefore, we have
\begin{equation}
    \left(\int_{\partial D_i} \phi^{\frac{2(n-1)}{n-2}} \mathrm{dvol}_{\partial D_i}\right)^{\frac{n-2}{2(n-1)}} \leqslant C \left(\int_{D_i}|\nabla \phi|^2\mathrm{dvol}_{ D_i}\right)^{\frac{1}{2}}.
\end{equation}
Combining this with \cite[Theorem 1.1]{Brendle21}, similar to arguments in \cite[Section 3.3]{CL23}, we obtain that
\begin{equation*}
    \left(\int_{D_i} \phi^{\frac{2n}{n-2}}\mathrm{dvol}_{\Sigma}\right)^{\frac{n-2}{n}}\leqslant C \int_{D_i}|\nabla \phi|^2\mathrm{dvol}_{ \Sigma}\leqslant C.
\end{equation*}
Since $\mathrm{Vol}(D_i)$ goes to infinity, it follows that if $u$ is a constant function, then $u\equiv 0$ or $u\equiv 1$. As argued in \cite[Lemma 2]{CSZ97}, we may assume that $u\equiv 1$. We choose a smooth function $\psi$ such that
\begin{equation}\label{psi}
    \psi=\left\{
\begin{array}{cl}
&1, ~~\mbox{in}~~ E^{(i_0)}_2\\
&0, ~~\mbox{in}~~ E^{(i_0)}_j, j\neq 2\\
\end{array}
\right.
\end{equation}
and 
\begin{equation*}
    |\nabla \psi|<C, \quad 0\leqslant \psi \leqslant 1
\end{equation*}
for some constant $C$ independent of $i$ and $u_i$. Note that $|\nabla \psi|$ vanishes outside a compact set.

By (\ref{DNP}) and (\ref{psi}), the function $\psi u_i$ vanishes on on $\partial E^{(i)}_1\cup \partial E^{(i)}_2$. Similar to above, we can claim that 
\begin{equation*}
    \left(\int_{D_i} (\psi u_i)^{\frac{2n}{n-2}}\mathrm{dvol}_{\Sigma}\right)^{\frac{n-2}{n}}\leqslant C. 
\end{equation*}
In particular, 
\begin{equation*}
    \left(\int_{D_i\cap E^{(i_0)}_2} (\psi u_i)^{\frac{2n}{n-2}}\mathrm{dvol}_{\Sigma}\right)^{\frac{n-2}{n}}\leqslant C, \quad ~~\mbox{for all}~~ i\geqslant i_0. 
\end{equation*}
Letting $i\to \infty$, we get $\mathrm{Vol}(E^{(i_0)}_2)\leqslant C$, a contradiction with Lemma \ref{volume}. Therefore, the limiting hamronic function is not a constant function.
\end{proof}
\begin{lemma}\label{estimate}
    Suppose that $\left(\Sigma^3, \left\{\partial_3 \Sigma\right\}\right)\subset \left(\Omega^{4}, \left\{\partial_{4} \Omega\right\}\right)$ is a stable $C^{1,1}$-to-edge properly embedded FBMH and u is a harmonic function on $\Sigma$. Then
    \begin{equation*}
    \left(1-\frac{1}{\sqrt{2}}\right)\int_{\Sigma} \varphi^2 |A|^2|\nabla u|^2 \mathrm{dvol}_{\Sigma}+\frac{1}{2}\int_{\Sigma} \varphi^2 |\nabla|\nabla u||^2 \mathrm{dvol}_{\Sigma}\leqslant \int_{\Sigma} |\nabla\varphi|^2 |\nabla u|^2 \mathrm{dvol}_{\Sigma},
\end{equation*}
for any $\varphi\in C^1(\Sigma)$.
\end{lemma}
\begin{proof}
    Detecting the proof of \cite[Lemma 21]{CL23}, we find that there is an extra boundary integral in (12) when integrating by parts; i.e, we need to estimate
    \begin{equation*}
        \int_{\partial \Sigma}|\nabla u|\langle \nabla|\nabla u|, \eta\rangle \varphi^2 \mathrm{dvol}_{\partial \Sigma}. 
    \end{equation*}
Since $\mathcal{H}^{n-1}(\partial^{E}\Sigma)=0$, it suffices to show that $\langle \nabla|\nabla u|, \eta\rangle(p)=0$ for all $p\in \partial^F \Sigma$. Note that the harmonic function $u$ satisfies the Neumann condition $\frac{\partial u}{\partial \eta}=0$ on $\partial^F \Sigma$. Let $\tilde{u}$ be the even reflection of $u$ across $\partial H_{+}$. Then $\tilde{u}$ is a harmonic function in a neighborhood of $p$. Now let $\gamma:(-1,1)\to \tilde{\Sigma}$ be a curve that
\begin{equation*}
    \gamma(0)=p,\quad \gamma^{'}(0)=\eta,\quad  \gamma(t)=R(\gamma(-t)),
\end{equation*}
where $R$ denotes reflection across $\partial H_{+}$. Then $|\nabla \tilde{u}|(\gamma(t))$ is an even function of $t$ and so
\begin{equation}
    0=\frac{d}{d t}\bigg|_{t=0}|\nabla \tilde{u}|(\gamma(t))= \langle \nabla|\nabla u|, \eta\rangle(p).
\end{equation}
Since $p$ is arbitrary, the boundary integral therefore vanishes. The remainder of the proof can now be completely exactly as in \cite{CL23}.
\end{proof}
Now we are ready to prove the one-endness.
\begin{proposition}\label{end}
Suppose that $\left(\Sigma^3, \left\{\partial_3 \Sigma\right\}\right)\subset \left(\Omega^4, \left\{\partial_4 \Omega\right\}\right)$ is a stable $C^{1,1}$-to-edge properly embedded FBMH. $\Sigma$ has only one end.    
\end{proposition}
\begin{proof}
    Suppose the contrary, that $\Sigma$ has at least two ends. Then Lemma \ref{harmonic} implies that $\Sigma$ admits a nontrivial harmonic function $u$ with finite energy. Take an arbitrary point $p\in \mathring{\Sigma}$. For $\rho >0$, take $\varphi \in C^1(\Sigma)$ such that 
\begin{equation}
    \varphi=\left\{
\begin{array}{cl}
&1, ~~\mbox{in}~~ B_{\Sigma}(p, \rho)\\
&0, ~~\mbox{in}~~ B^c_{\Sigma}(p,2\rho),\\
\end{array}
\right.
\end{equation}
 and $|\nabla \varphi|\leqslant \frac{2}{\rho}$. Note that the constructed $\varphi$ are not needed to be compact supported in $\mathring{\Sigma}$. Then Lemma \ref{estimate} implies that  
 \begin{equation*}
    \left(1-\frac{1}{\sqrt{2}}\right)\int_{B_{\Sigma}(p, \rho)} \varphi^2 |A|^2|\nabla u|^2 \mathrm{dvol}_{\Sigma}+\frac{1}{2}\int_{B_{\Sigma}(p, \rho)} \varphi^2 |\nabla|\nabla u||^2 \mathrm{dvol}_{\Sigma}\leqslant \frac{4C}{\rho^2}.
\end{equation*}
 Sending $\rho\to \infty$, we conclude that $|\nabla|\nabla u||^2\equiv 0$, which implies that $|\nabla u|$ is a constant. Since $u$ is nonconstant, we have that $|\nabla u|>0$. However, this implies that
 \begin{equation*}
     \mathrm{Vol}(\Sigma)=\frac{1}{|\nabla u|^2}\int_{\Sigma}|\nabla u|^2<\infty,
 \end{equation*}
 a contradiction.
\end{proof}

\section{A conformal deformation of metrics}
Suppose that $\left(\Sigma^n, \left\{\partial_n \Sigma\right\}\right)\subset \left(\Omega^{n+1}, \left\{\partial_{n+1} \Omega\right\}\right)$ is a stable $C^{1,1}$-to-edge properly embedded FBMH. Without loss of generality, we may assume that $0\in \partial^{E}\Sigma^n$. In this section we carry out the conformal deformation technique used by Chodosh--Li \cite{CL23}. 

Consider the function $r(x)=\mathrm{dist}_{\R^{n+1}}(0,x)$ on $\Sigma$. As shown in \cite[Section 5]{CL23}, we find that
\begin{equation*}
    \Delta r=\frac{n}{r}-\frac{|\nabla r|^2}{r}.
\end{equation*}

Suppose that $w>0$ is a smooth function on $\Sigma\backslash \{0\}$. On $\Sigma\backslash \{0\}$, define $\tilde{g}=w^2g$. For $\lambda\in \R$, we consider the quadratic form
\begin{equation*}
    \mathcal{Q}_{w}(\varphi)=\int_{\Sigma}\left(|\tilde{\nabla}\varphi|^2+\left(\frac{1}{2}\tilde{R}-\lambda\right)\varphi^2\right)dV_{\tilde{g}},
\end{equation*}
where $\varphi \in C^1_c (\Sigma\backslash \{0\})$ satisfies
\begin{equation}\label{Neu}
    \frac{\partial \varphi}{\partial \tilde{\eta}}=0, \quad ~~\mbox{on}~~ \partial \left(\Sigma\backslash \{0\}\right),
\end{equation}
and
$\tilde{\nabla}$, $\tilde{R}$, $dV_{\tilde{g}}$, $\tilde{\eta}$ are the gradient, the scalar curvature, the volume form and the outward normal derivative with respect to $\tilde{g}$, respectively. In the sequel, we choose $w=r^{-1}$ on $\Sigma\backslash \{0\}$, and the dimension of $\Sigma^n$ is equal to $3$.  

One relates the geometric quantities in $g$ and $\tilde{g}$ as follows:
\begin{equation}\label{conf}
    |\nabla\varphi|_g^2=w^2|\tilde{\nabla}\varphi|_{\tilde{g}}^2, \quad dV_{\tilde{g}}=w^{-n}dV_{g}, \quad \tilde{\eta}=w^{-1}\eta.
\end{equation}
Moreover, we have
\begin{equation*}
    w^2\tilde{R}=R-2(n-1)\Delta \log w-(n-1)(n-2)|\nabla \log w|^2.
\end{equation*}
For $\varphi \in C^1_c (\Sigma\backslash \{0\})$ satisfying (\ref{Neu}), by (\ref{conf}), we know that
\begin{equation}
     \frac{\partial \varphi}{\partial \eta}=0.
\end{equation}
Note that the radial function $w$ can be extended to be an even function on the reflection of $\Sigma$. Combining it with the argument used in Section 3, we have 
\begin{equation}\label{zero}
    \frac{\partial (w^{\frac{2-n}{2}}\varphi)}{\partial \tilde{\eta}}=w^{\frac{2-n}{2}}\frac{\partial \varphi}{\partial \tilde{\eta}}+\varphi\frac{\partial w^{\frac{2-n}{2}}}{\partial \tilde{\eta}}\\
    =\varphi\frac{\partial w^{\frac{2-n}{2}}}{\partial \tilde{\eta}}=\varphi w^{-1}\frac{\partial w^{\frac{2-n}{2}}}{\partial {\eta}}=0.
\end{equation}

Denote by $\tilde{\mathcal{Q}}_w(\varphi):=\mathcal{Q}_{w}(w^{\frac{2-n}{2}}\varphi)$. Following the same computation in \cite[Section 5]{CL23}, we have
\begin{equation*}
    \begin{split}
    &\tilde{\mathcal{Q}}_w(\varphi)\\
    =&\int_{\Sigma}\left(w^{-2}|{\nabla}(w^{\frac{2-n}{2}}\varphi)|_g^2+\left(\frac{1}{2}\tilde{R}-\lambda\right)w^{2-n}\varphi^2\right)w^n dV_{{g}}\\
    =&\int_{\Sigma}\left(|{\nabla}\varphi-\frac{n-2}{2}\varphi \nabla \log w|_g^2+\left(\frac{1}{2}\tilde{R}-\lambda\right)w^{2}\varphi^2\right) dV_{{g}}\\
    =&\int_{\Sigma}\left(|{\nabla}\varphi|_g^2-\frac{n-2}{2}\langle {\nabla}\varphi^2, \nabla \log w\rangle_g+\left(\frac{n-2}{2}\right)^2 |\nabla \log w|^2_g\varphi^2+\left(\frac{1}{2}\tilde{R}-\lambda\right)w^{2}\varphi^2\right) dV_{{g}}\\
    =&\int_{\Sigma}\left(|{\nabla}\varphi|_g^2+\left(\frac{n-2}{2}\Delta \log w+\left(\frac{n-2}{2}\right)^2 |\nabla \log w|^2_g+\left(\frac{1}{2}\tilde{R}-\lambda\right)w^{2}\right)\varphi^2\right) dV_{{g}}\\
    &-\int_{\partial \Sigma}\frac{n-2}{2}\varphi^2 \frac{\partial \log w}{\partial \eta}dA_g\\
    =&\int_{\Sigma}\left(|{\nabla}\varphi|_g^2+\frac{1}{2}{R}\varphi^2-\left(\frac{n}{2}\left(\Delta \log w+\frac{n-2}{2} |\nabla \log w|^2_g\right)+\lambda w^{2}\right)\varphi^2\right) dV_{{g}}
    \end{split}
\end{equation*}
where the boundary term vanishes due to (\ref{zero}). Note that
\begin{equation*}
    \Delta \log w+\frac{n-2}{2} |\nabla \log w|^2_g=-\frac{n}{r^2}+\frac{n+2}{2}\frac{|\nabla r|^2}{r^2}.
\end{equation*}
Therefore,
\begin{equation}
\begin{split}
\tilde{\mathcal{Q}}_w(\varphi)&=\int_{\Sigma}\left(|{\nabla}\varphi|_g^2+\frac{1}{2}{R}\varphi^2+\left(\frac{n}{2}\left(n-\frac{n+2}{2} |\nabla r|^2_g\right)-\lambda \right)r^{-2}\varphi^2\right) dV_{{g}}\\
  &\geqslant \int_{\Sigma}\left(|{\nabla}\varphi|_g^2+\frac{1}{2}{R}\varphi^2+\left(\frac{n(n-2)}{4}-\lambda \right)r^{-2}\varphi^2\right) dV_{{g}}.
    \end{split}
\end{equation}
By the Gauss equation, we have $|A|_g^2+R_g=0$. On the other hand, by \cite[Theorem 5.1]{MW23}, we know that for any $\varphi \in C^1_c (\Sigma\backslash \{0\})$,
\begin{equation*}
    \int_{\Sigma}\left(|{\nabla}\varphi|_g^2+\frac{1}{2}{R}\varphi^2\right)dV_{g}\geqslant \int_{\Sigma}\left(|{\nabla}\varphi|_g^2-|A|^2_g\varphi^2\right)dV_{g}\geqslant 0.
\end{equation*}
We summarize these in the following Proposition. 
\begin{proposition}
    Suppose that $\left(\Sigma^n, \left\{\partial_n \Sigma\right\}\right)\subset \left(\Omega^{n+1}, \left\{\partial_n \Omega\right\}\right)$ is a stable $C^{1,1}$-to-edge properly embedded FBMH. Then the conformally deformed manifold $(\Sigma\backslash \{0\}, \tilde{g}=r^{-2}g)$ satisfies 
    \begin{equation}
        \lambda^N_1\left(-\tilde{\Delta}+\frac{1}{2}\tilde{R}\right)\geqslant \lambda,
    \end{equation}
where $\lambda^N_1$ is the first Neumann eigenvalue and $\lambda=\frac{n(n-2)}{4}$.
\end{proposition}

\section{Volume Estimate}
We first prove a diameter bound for free boundary warped $\mu-$bubbles in $3-$manifolds with boundary satisfying $\lambda^N_1\left(-\tilde{\Delta}+\frac{1}{2}\tilde{R}\right)\geqslant \lambda>0$. 

Recall that $r(x)=\mathrm{dist}_{\R^{4}}(0,x)$ and $\bar{r}(x)=\mathrm{dist}_{(\Sigma,g)}(x,0)$, and we consider $\tilde{g}=r^{-2}g$. Combining the reflection technique used above and the fact that the conformal factor is symmetric with respect to the reflection, we conclude that $\partial^{F}_{+}\Sigma$ and $\partial^{F}_{-}\Sigma$ are totally geodesic with respect to $\tilde{g}$ as well. 

Fix $\rho>0$, and consider the ball $B_{\R^4}(0,e^{\frac{5\pi}{\sqrt{\lambda}}}\rho)$. By Proposition \ref{end}, $\Sigma\backslash B_{\R^4}(0,e^{\frac{5\pi}{\sqrt{\lambda}}}\rho)$ has only one unbounded connected component $E$. Denote by $\Sigma^{'}=\Sigma\backslash E$. We claim that $\partial \Sigma^{'}\cap \mathring{\Sigma}=\partial E\cap \mathring{\Sigma}$ is connected. Indeed, since $\Sigma^{'}$ and $E$ are both connected, if $\partial \Sigma^{'}\cap \mathring{\Sigma}$ has more than one connected components, one can find a loop in $\mathring{\Sigma}$ intersecting one component of $\partial \Sigma^{'}\cap \mathring{\Sigma}$ exactly once, contradicting that $\Sigma$ is simply connected.
For convenience, we denote $\partial \Sigma^{'}\cap \mathring{\Sigma}$, $\partial \Sigma^{'}\cap \partial H^4_{+}$ and $\partial \Sigma^{'}\cap \partial H^4_{-}$ by $\partial \mathring{\Sigma}^{'}$, $\partial^F_{+} {\Sigma}^{'}$ and $\partial^F_{-} {\Sigma}^{'}$, respectively.
\begin{lemma}\label{bubble}
    Let $(\Sigma^{'}, \tilde{g})$ be constructed as above satisfying
    \begin{equation}\label{spectrum}
        \lambda^N_1\left(-\tilde{\Delta}+\frac{1}{2}\tilde{R}\right)\geqslant \lambda>0.
    \end{equation}
    Then there exists a connected proper embedded open set $\Pi$ containing $\partial \mathring{\Sigma}^{'}$, $\Pi\subset B_{\Sigma}(\partial \mathring{\Sigma}^{'}, \frac{5\pi}{\sqrt{\lambda}})$, such that each connected component of $\partial \Pi\backslash \partial \mathring{\Sigma}^{'}$ is a $2$-sphere with area at most $\frac{8\pi}{\lambda}$ and intrinsic diameter at most $\frac{2\pi}{\sqrt{\lambda}}$.
\end{lemma}
\begin{proof}
    This is an application of estimates for the free boundary warped $\mu-$bubbles (see, e.g. \cite[Section 4]{CL23}). Since $\Sigma^{'}$ satisfies (\ref{spectrum}), there exists $u\in C^{\infty}(\mathring{\Sigma}^{'})$ satisfying (\ref{Neu}), $u>0$ in $\mathring{\Sigma}^{'}$ and $C^{1,\alpha}$-to-edge, such that
    \begin{equation}\label{scalar}
        \tilde{\Delta}_{\Sigma^{'}} u\leqslant -\frac{1}{2}\left(2\lambda-\tilde{R}_{\Sigma^{'}}\right)u.
    \end{equation}
Take $\varphi_0\in C^{\infty}(\Sigma^{'})$ to be a smoothing of $\mathrm{dist}_{\Sigma^{'}}(\cdot, \partial \mathring{\Sigma}^{'})$ such that $|\mathrm{Lip}(\varphi_0)|\leqslant 2$, and $\varphi_0=0$ on $\partial \mathring{\Sigma}^{'}$. Choose $\epsilon\in (0,\frac{1}{2})$ such that $\epsilon$, $\frac{4\pi}{\sqrt{\lambda}}+2\epsilon$ are regular values of $\varphi_0$. Define
\begin{equation*}
    \varphi=\frac{\varphi_0-\epsilon}{\frac{4}{\sqrt{\lambda}}+\frac{\epsilon}{\pi}}-\frac{\pi}{2},
\end{equation*}
$\Pi_1=\{x\in \Sigma^{'}:-\frac{\pi}{2}<\varphi<\frac{\pi}{2}\}$, and $\Pi_0=\{x\in \Sigma^{'}:-\frac{\pi}{2}<\varphi \leqslant 0\}$. We have that $|\mathrm{Lip}(\varphi)|\leqslant \frac{\sqrt{\lambda}}{2}$. In $\Pi_1$, define $h(x)=-\frac{1}{2}\tan (\varphi(x))$. By a direct computation, we have
\begin{equation}\label{h}
    \lambda+h^2-2|\tilde{\nabla} h|\geqslant 0.
\end{equation}
Note that $\partial \mathring{\Sigma}^{'}\cap \partial^F_{\pm} {\Sigma}^{'}$ consisting of smooth $1$-dimensional closed submanifolds. Moreover, we may assume that $\partial^F_{\pm} {\Sigma}^{'}$ meets $\partial \mathring{\Sigma}^{'}$ orthogonally.

Consider the functional
\begin{equation*}
    \mathcal{A}(\Pi)=\int_{\partial \Pi} ud\mathcal{H}^2-\int_{\Pi_1}\left(\chi_{\Pi}-\chi_{\Pi_0}\right)hu d\mathcal{H}^3
\end{equation*}
among Caccioppoli sets $\Pi$ in $\Pi_1$ with $\Pi\Delta \Pi_0$ compactly contained in $\Pi_1$. By \cite[Proposition 15]{CL23}, there exists $\overline{\Pi}$ with $\partial \overline{\Pi}\subset \Pi_1\cup \partial \mathring{\Sigma}^{'}$ minimizing $\mathcal{A}$ among such regions. $\partial \overline{\Pi}\cap \mathring{\Sigma}^{'}$ is smooth and along it,
\begin{equation}\label{H}
    H=-u^{-1}\langle \tilde{\nabla}_{\Sigma^{'}}u, \nu_{\partial \Pi}\rangle+h.
\end{equation}
Moreover, $\partial \overline{\Pi}$ meets $\partial^F_{\pm} {\Sigma}^{'}$ orthogonally and it may have nonempty intersection with $\partial^E \Sigma^{'}$.

We take $\Pi$ to be the connected component of $\{x\in \Sigma^{'}:0\leqslant\varphi_0 \leqslant \epsilon\}\cup \overline{\Pi}$ that contains $\partial \mathring{\Sigma}^{'}$.  We claim that each connected component $\Xi$ of $\partial \Pi\cap \Pi_1$ is properly embedded in $\Sigma^{'}$. Without loss of generality, we may assume that $\partial \Xi$ is not contained in $\mathring{\Sigma}^{'}$. First, by (\ref{Neu}) and the fact that $\partial^F_{\pm} {\Sigma}^{'}$ are totally geodesic, we know that $\mathring{\Xi}$ can't touch $\partial^F_{\pm} {\Sigma}^{'}$. Besides, at any $p\in \partial^F \Xi \cap \partial^E \Sigma^{'}$, without loss of generality, we may assume that the unit outward co-normal $\varsigma$ to $\partial^F \Xi$ is orthogonal to $\partial^F_{+} {\Sigma}^{'}$. Denote $X\in T_p \partial^F_{-} {\Sigma}^{'}$ the unit vector orthogonal to $T_p \partial^F \Xi$. Note that by our assumption, the inner product $\langle \varsigma, X\rangle$ at the point $p$ is negative. We can extend $X$ to be a variation around $p$ which preserves the negativity of $\langle \varsigma, X\rangle$. Combining this with (\ref{H}) and the first variation formula, we have
\begin{equation*}
    \delta\mathcal{A}_{\Pi}(X)=\int_{\partial^F \Xi}\langle \varsigma, X\rangle<0,
\end{equation*}
which is contradicted with the minimizing property of $\Pi$.

Due to the $C^{1,\alpha}$-to-edge regularity, we shall make use of the log-cutoff trick. Specifically, we define cutoff functions $\eta_r$ on $\Pi$ which vanish in a neighborhood of $\partial^E \Pi$ (possibly empty) and which convergence pointwisely to $1$ on $\Pi\backslash \partial^E\Pi$ as $r\to 0$. Since $\partial^E \Pi$ is of codimension 2 in $\Pi$, we can arrange that 
\begin{equation*}
    \int_{\Pi} |\nabla \eta_r|^2 d\mathcal{H}^2\to 0, ~~\mbox{as}~~ r\to 0.
\end{equation*}

Now, we verify that $\Pi$ satisfies the conclusions of this Lemma. Given a test function $\psi$ defined on $\Pi$, composing with the cutoff function $\eta_r$ if necessary. For convenience, we still denote it by $\psi$. Indeed, for any connected component $\Xi$ of $\partial \Pi\cap \Pi_1$, the stability of $\mathcal{A}$ \cite[Proposition]{CL23} implies
\begin{equation*}
    \begin{split}
        0\leqslant&\int_{\Xi}\left(|\nabla \psi|^2u-\frac{1}{2}\left(\tilde{R}_{\Sigma^{'}}-\lambda-\tilde{R}_{\Xi}+|\mathring{\tilde{A}}_{\Xi}|^2\right)\psi^2 u+\left(\tilde{\Delta}_{\Sigma^{'}}u-\tilde{\Delta}_{\Xi}u\right)\psi^2\right)d\mathcal{H}^2\\
        &\int_{\Xi}\left(-\frac{1}{2}u^{-1}\langle \tilde{\nabla}_{\Sigma^{'}}u, \nu_{\partial \Pi}\rangle^2 \psi^2-\frac{1}{2}\left(\lambda+h^2+2\langle \tilde{\nabla}_{\Sigma^{'}} h, \nu_{\partial \Pi}\right)\psi^2 u\right)d\mathcal{H}^2\\
        &-\int_{\partial \Xi} \tilde{A}_{\partial^F_{\pm} {\Sigma}^{'}}(\nu_{\partial \Pi}, \nu_{\partial \Pi}) \psi^2 u d\mathcal{H}^1.
    \end{split}
\end{equation*}
Using (\ref{scalar}) and (\ref{h}), we have
\begin{equation*}
    \begin{split}
        0\leqslant&\int_{\Xi}\left(|\nabla \psi|^2u+\frac{1}{2}\left(\tilde{R}_{\Xi}-\lambda\right)\psi^2 u-\tilde{\Delta}_{\Xi}u\psi^2\right)d\mathcal{H}^2-\int_{\partial \Xi} \tilde{A}_{\partial^F_{\pm} {\Sigma}^{'}}(\nu_{\partial \Pi}, \nu_{\partial \Pi}) \psi^2 u d\mathcal{H}^1.
    \end{split}
\end{equation*}
Taking $\psi=u^{-\frac{1}{2}}$ and integrating by parts, we have
\begin{equation}
\begin{split}
    0&\leqslant\int_{\Xi}\left(\frac{1}{4}|\nabla u|^2u^{-2}+\tilde{K}_{\Xi}-\frac{1}{2}\lambda-u^{-1}\tilde{\Delta}_{\Xi}u\right)d\mathcal{H}^2-\int_{\partial \Xi} \tilde{A}_{\partial^F_{\pm} {\Sigma}^{'}}(\nu_{\partial \Pi}, \nu_{\partial \Pi}) \psi^2 u d\mathcal{H}^1\\
    &=\int_{\Xi}\left(-\frac{3}{4}|\nabla u|^2u^{-2}+\tilde{K}_{\Xi}-\frac{1}{2}\lambda\right)d\mathcal{H}^2-\int_{\partial \Xi}\left( \tilde{A}_{\partial^F_{\pm} {\Sigma}^{'}}(\nu_{\partial \Pi}, \nu_{\partial \Pi})+u^{-1}\frac{\partial u}{\partial \tilde{\eta}}\right) d\mathcal{H}^1.
    \end{split}
\end{equation}
Recall that $\partial^{F}_{+}\Sigma$ and $\partial^{F}_{-}\Sigma$ are totally geodesic with respect to $\tilde{g}$ and $u$ satisfies the Neumann condition (\ref{Neu}). Boundary integrals in above vanish, and we conclude that
\begin{equation*}
    \lambda |\Xi|\leqslant 2\int_{\Xi }\tilde{K}_{\Xi} d\mathcal{H}^2 \leqslant 8\pi \Rightarrow |\Xi|\leqslant \frac{8\pi}{\lambda}.
\end{equation*}
Note that we have used Gauss--Bonnet formula, which also implies that $\Xi$ is a $2$-sphere. The diameter upper bound follows from \cite[Lemma 17]{CL23}.
\end{proof}

\begin{proof}[Proof of Theorem \ref{main}]
First of all, we consider $\lceil \frac{\pi}{\theta}\rceil$ times reflections of $\Sigma$ with respect to $\partial H^4_{+}$ or $\partial H^4_{-}$. The union of these manifolds is still a stable $C^{1,1}$-to-edge properly embedded free boundary minimal hypersurface in a wedge domain with angle $\lceil \frac{\pi}{\theta}\rceil \cdot \theta$. By the assumption on $\theta$, we know that $\pi<\lceil \frac{\pi}{\theta}\rceil \cdot \theta<2\pi$. For convenience, we denote this union still by $\Sigma$.
   
   Applying Lemma \ref{bubble} to $(\Sigma^{'}\backslash \{0\}, \tilde{g})$, we find a connected open set $\Pi$ in the $\frac{5\pi}{\sqrt{\lambda}}$ neighborhood of $\partial \mathring{\Sigma}^{'}$, such that each connected component of $\partial \Pi\backslash \partial \mathring{\Sigma}^{'}$ is a $2-$sphere with area at most $\frac{8\pi}{\lambda}$ and intrinsic diameter at most $\frac{2\pi}{\sqrt{\lambda}}$. Let $\Sigma_0$ be the connected component of $\Sigma^{'}\backslash \Pi$ that contains $0$.

   We make a few observations about $\Sigma_0$. First, we claim that $\Sigma\backslash \Sigma_0$ is connected. To see this, let $\Sigma_1$ be the union of connected components of $\Sigma^{'}\backslash \Pi$ other than $\Sigma_0$. Then $\Sigma\backslash \Sigma_0=\Sigma_1\cup \Pi\cup E$. Note that each connected component of $\Sigma_1$ shares a common boundary with $\Pi$. Since $\Pi$ is connected, so is $\Sigma_1\cup \Pi$. Next, we claim that $\Sigma_0\cap \mathring{\Sigma}$ has only one boundary component: otherwise, since both $\Sigma_0$ and $\Sigma\backslash \Sigma_0$ are connected, as before we can find a loop in $\mathring{\Sigma}$ intersecting one component of $\partial \Sigma_0\cap \mathring{\Sigma}$ exactly once, contradicting that $\Sigma$ is simply connected.

   Denote by $\Xi=\partial \Sigma_0\cap \mathring{\Sigma}$. By (1) in \cite[Lemma 25]{CL23}, we know that $\min_{x\in \Xi} r(x)\geqslant \rho$. Since $\Sigma\cap B_{\R^4}(0,\rho)$ is connected, this implies that $\left(\Sigma\cap B_{\R^4}(0,\rho)\right)\subset \Sigma_0$. Obviously, $\max_{x\in \Xi} r(x)\leqslant e^{\frac{5\pi}{\sqrt{\lambda}}}\rho$. Therefore, we have
   \begin{equation*}
       |\Xi|_g=\int_{\Xi}dA_g=\int_{\Xi} r^2 dA_{\tilde{g}}\leqslant e^{\frac{10\pi}{\sqrt{\lambda}}}\rho^2 |\Xi|_{\tilde{g}}\leqslant \frac{8\pi}{\lambda} e^{\frac{10\pi}{\sqrt{\lambda}}}\rho^2.
   \end{equation*}
   By our construction, the complement of the wedge domain with angle $\lceil \frac{\pi}{\theta}\rceil \cdot \theta$ is a convex body in $\R^4$. Note that $\Xi$ is the relative boundary of $\Sigma_0$, and $\partial^F_{+} {\Sigma}_0:=\partial \Sigma_0\cap \partial H^4_{+}$, $\partial^F_{-} {\Sigma}_0:=\partial \Sigma_0\cap \partial H^4_{-}$ are free boundaries of $\Sigma_0$. By \cite[Theorem 1.2]{LWW23}, we have
   \begin{equation*}
       |\Sigma\cap B_{\R^4}(0,\rho)|_g\leqslant |\Sigma_0|_g\leqslant \frac{4 |\B^3|_{g_c}}{|\partial \B^3|^{\frac{3}{2}}_{g_c}} |\Xi|_{g}^{\frac{3}{2}},
   \end{equation*}
   where $\B^3$ is the unit round ball in $(\R^3, g_c)$ with canonical metric $g_c$. Noting that the union contains the original free boundary minimal hypersurface, we complete the proof. 
\end{proof}

\bibliography{mybib}
\bibliographystyle{alpha}

\end{document}